\documentclass[draft]{amsart}
\usepackage{amsfonts}
\usepackage{latexsym}
\usepackage{amssymb}
\usepackage{amsmath}
\usepackage{amsfonts, color}

\newcommand{\R}{\mathbb R}
\newcommand{\N}{\mathbb N}

\newcommand{\E}{\mathbb E}
\newcommand{\Pro}{\mathbb P}

\newcommand{\conv}{{\rm conv}}
\newcommand{\inter}{{\rm int}}

\newcommand{\inner}[2]{ \ensuremath { \langle {#1},{#2} \rangle } }

\newtheorem{thm}{Theorem}[section]

\newtheorem{lemma}{Lemma}[section]

\newtheorem{proposition}{Proposition}[section]
\theoremstyle{definition}
\newtheorem{rmk}{Remark}[section]

\begin{document}


\title{On mean outer radii of random polytopes}

\author{D. Alonso-Guti\'{e}rrez}
\address{Departamento de Matem\'aticas, Universidad de Murcia, Campus de
Espinar\-do, 30100-Murcia, Spain} \email{davidalonso@um.es; mhcifre@um.es}

\author{N. Dafnis}
\address{Department of Mathematics, Texas A$\&$M University, College
Station, TX 77843 U.S.A.} \email{nikdafnis@gmail.com}

\author{M. \'{A}. Hern\'andez Cifre}

\author{J. Prochno}
\address{Institute of Analysis, Johannes Kepler University Linz,
Altenbergerstr. 69, A-4040 Linz, Austria} \email{joscha.prochno@jku.at}


\thanks{Part of this work was done while the first, second and fourth
authors were postdoctoral fellows at the Department of Mathematical and
Statistical Sciences at the University of Alberta in Edmonton under the
supervision of N. Tomczak-Jaegermann and A. Litvak. Second author is
supported by the Action ``Supporting Postdoctoral Researchers'' of the
Operational Program ``Education and Lifelong Learning'' (Action's
Beneficiary: General Secretariat for Research and Technology), and is
co-financed by the European Social Fund (ESF) and the Greek State. First
and third authors are partially supported by MICINN project MTM2010-16679,
MICINN-FEDER project MTM2009-10418 and ``Programa de Ayudas a Grupos de
Excelencia de la Regi\'on de Murcia'', Fundaci\'on S\'eneca,
04540/GERM/06. Fourth author is supported by the Austrian Science Fund (FWF) within the
FWF project P23987 ``Projection operators in
Analysis and geometry of classical Banach spaces''.}

\begin{abstract}
In this paper we introduce a new sequence of quantities
for random polytopes. Let $K_N=\conv\{X_1,\dots,X_N\}$ be a random polytope generated by
independent random vectors uniformly distributed in an isotropic convex
body $K$ of $\R^n$. We prove that the so-called $k$-th mean outer radius
$\widetilde R_k(K_N)$ has order $\max\{\sqrt{k},\sqrt{\log N}\}L_K$ with
high probability if $n^2\leq N\leq e^{\sqrt{n}}$. We also show that this
is also the right order of the expected value of $\widetilde R_k(K_N)$ in
the full range $n\leq N\leq e^{\sqrt{n}}$.
\end{abstract}

\maketitle

\section{Introduction and Notation}

The study of random polytopes began with Sylvester and the famous
four-point problem nearly 150 years ago and it represents the starting
point for an extensive study. In 1963, R\'enyi and Sulanke \cite{RSu}
continued it in their innovative work, studying expectations of various
basic functionals of random polytopes. Important quantities are
expectations, variances and distributions of those functionals. The tools
combine a variety of mathematical disciplines such as convex geometry,
geometric analysis or geometric probability (see also \cite{Ba,Re} and the
references therein).

Random polytopes appear in many important applications and other fields,
{\it e.g.}, in statistics (extreme points of random samples), convex
geometry (approximation of convex sets), but also in computer science in
the analysis of the average complexity of algorithms \cite{PSh},
optimization \cite{B}, and even in biology \cite{S}. These are several
reasons why in the last 30 years random polytopes have gained more and
more interest.

Important and very recent developments in the study of random polytopes
can be found in \cite{KK}, where the authors prove that the isotropic
constant of a Gaussian random polytope is bounded with high probability,
and in \cite{A,AP1,AP2,DGG,DGT1,DGT2,PP} where the authors study the
relation between several parameters of a random polytope in an isotropic
convex body and the isotropic constant of the body.

The main goal of this paper is to estimate a new sequence of quantities
for random polytopes defined by $N$ random points in an isotropic convex
body in $\R^n$, namely, a certain family of outer radii. In the last
years, the so-called successive outer radii have been intensively studied,
{\it e.g.} their geometric and analytical properties, size for special
bodies, relation with other measures, and computational aspects (see {\it
e.g.} \cite{BH92,Br05,GHC,GHCH,GK2,HHC,HHC2,Pu} and the references
inside). These families of quantities are defined as the maximum/minimum
of the (classical) circumradii of the projections of a convex body onto
all $k$-dimensional subspaces. If we restrict to the family of polytopes,
there is not too much information (see {\it e.g.} \cite{BrT06,GK2} and, in
particular, \cite{Br05}, where the precise values of -among others- the
successive outer radii of the regular cube, cross-polytope and simplex can
be found).

So, it is a natural generalization to consider a kind of mean outer radii.
In the context of random polytopes, we will define them as the outer
radius of the projection of the random polytope onto a subspace $F$ of
dimension $k$, averaged over the Grassmannian manifold $G_{n,k}$, $1\leq k
\leq n$, with respect to the Haar probability measure, and we will prove
that, with high probability, this quantity has order
$\max\{\sqrt{k},\sqrt{\log N}\}L_K$, where $L_K$ is the isotropic constant
of $K$.

\medskip

We will work in $\R^n$ equipped with a Euclidean structure $\langle
\cdot,\cdot \rangle$. We denote by $|\cdot|$ the corresponding Euclidean
norm, as well as the volume ({\it i.e.}, the Lebesgue measure) in $\R^n$.
A {\it convex body} $K\subset\R^n$ is a compact convex set with non-empty
interior, and it is called ({\it centrally}) \textit{symmetric} if $-x\in
K$, whenever $x\in K$. We write $B_2^n$ to denote the Euclidean unit ball,
$S^{n-1}=\bigl\{x\in\R^n:|x|=1\bigr\}$ for the unit sphere in $\R^n$ and
$\sigma$ for the uniform probability measure on $S^{n-1}$. Moreover, if
$F\subset\R^n$ is a $k$-dimensional linear subspace, we denote by
$S^{k-1}_F$ and $\sigma_F$ the corresponding unit sphere and probability
measure in $F$, respectively.

We use the notation $a\simeq b$ to denote that there exist absolute
constants $c_1>0$, $c_2>0$ such that $c_1a\leq b\leq c_2 a$.

A convex body $K$ is \textit{isotropic} if $|K|=1$, its centroid lies at
the origin, {\it i.e.}, $\int_K x\,dx=0$, and it satisfies the
\textit{isotropic condition}
\[
\int_K\langle x,\theta\rangle^2dx=L_K^2,\quad\text{ for all }\;\theta\in
S^{n-1},
\]
where $L_K$ is a constant depending only on $K$, the so-called
\textit{isotropic constant} of $K$. We refer to \cite{Gi,MP} for further
background information and results on isotropic convex bodies.

For a given convex body $K\subset\R^n$ and independent random vectors
$X_1,\dots,X_N$, uniformly distributed in $K$, we denote the corresponding
random polytopes by
\[
K_N=\conv\{X_1,\dots,X_N\}\quad\text{and}\quad S_N=\conv\{\pm
X_1,\ldots,\pm X_N\}.
\]
Moreover, if $0\in\inter K$, the \textit{outer radius} of $K$ is defined
as the quantity
\[
R(K)=\min\left\{t>0:\,K \subseteq tB_2^n\right\}=\max_{x\in K} |x|.
\]
We observe that if $K$ is symmetric, then $R(K)$ coincides with the
classical {\it circumradius} of $K$, namely,
$\min\{R:\exists\,x\in\R^n\text{ with } K\subseteq x+RB_2^n\}$.

This magnitude leads to the definition of the main concept of this paper.
For a convex body $K\subseteq\R^n$, the {\it $k$-th mean outer radius} of
$K$, $1\leq k\leq n$, is defined as
\[
\widetilde{R}_k(K)=\int_{G_{n,k}}R(P_FK)\, d\nu_{n,k}(F),
\]
where $G_{n,k}$ denotes the set of $k$-dimensional linear subspaces of
$\R^n$, $\nu_{n,k}$ is the unique Haar probability measure on $G_{n,k}$
invariant under orthogonal maps, and $P_FK$ is the orthogonal projection
of $K$ onto $F$. We will write $d\nu_{F,k}$ if we work with the
Grassmannian manifold $G_{F,k}$ restricted to a fixed subspace $F$.
Clearly $\widetilde{R}_n(K)$ is the outer radius of $K$ and, moreover, if
$K$ is symmetric, then $\widetilde{R}_1(K)=\omega(K)$ is the mean width of
$K$ and $\widetilde{R}_n(K)$ its classical circumradius. We also observe
that for the polytopes $K_N$ and $S_N$ and every $F\in G_{n,k}$, $1\leq
k\leq n$, we have
\begin{equation*}\label{eq04}
R(P_F K_N)=R(P_F S_N)=\max_{1\leq j \leq N} |P_F X_j|.
\end{equation*}
Hence, all results in this paper are valid for both $K_N$ and $S_N$, and
so it suffices to deal only with one of them, say $K_N$.

\medskip

The main result of this paper is an asymptotic formula which gives the
right order for the mean outer radii of a random polytope lying in an
isotropic convex body, with high probability:

\begin{thm}\label{THM estimate high brob}
Let $K\subset\R^n$ be an isotropic convex body and $K_N$ be a random
polytope generated by independent random vectors $X_1,\dots,X_N$,
uniformly distributed in $K$. If $n\leq N\leq e^{\sqrt n}$ then, for all
$1\leq k\leq n$ and $s>0$,
\begin{equation}\label{EQ estimate high brob1}
c_1(s)\max\left\{\sqrt k,\sqrt{\log\frac{N}{n}}\right\}L_K \leq
\widetilde{R}_k(K_N)\leq c_2(s)\max\left\{\sqrt k,\sqrt{\log N}\right\}L_K
\end{equation}
with probability greater than $1-N^{-s}$, where $c_1(s)$ and $c_2(s)$ are
positive absolute constants depending only on $s$. In particular, if
$n^2\leq N\leq e^{\sqrt n}$, then both the upper and the lower bounds have
the same order, {\it i.e.},
\begin{equation}\label{EQ estimate high brob2}
\widetilde{R}_k(K_N)\simeq\max\left\{\sqrt k,\sqrt{\log N}\right\}L_K.
\end{equation}
\end{thm}

In order to prove Theorem \ref{THM estimate high brob}, we modify some
arguments appearing in \cite{DGT1,DGT2}, and use in a crucial way two
strong inequalities of Paouris (see \cite{P1,P2}). We observe (see Remark
\ref{rem_extending}) that we cannot expect to extend the range for $N$
from above in estimate \eqref{EQ estimate high brob2} with high
probability. On the other hand, even though we can extend this range of
$N$ from below, we cannot obtain this estimate with high probability when
$N$ is proportional to $n$ with the techniques we use. However, this is
feasible if the expectation is involved:
\begin{thm}\label{THM estimate expectation}
Let $K\subset\R^n$ be an isotropic convex body and $K_N$ be a random
polytope generated by independent random vectors $X_1,\dots,X_N$,
uniformly distributed in $K$. Then, for $n\leq N\leq e^{\sqrt n}$ and all
$1\leq k\leq n$,
\begin{equation*}
\E\widetilde{R}_k(K_N)\simeq\max\left\{\sqrt k,\sqrt{\log N}\right\}L_K.
\end{equation*}
\end{thm}

We mention here that using a new method, in \cite{AP2} the authors
established a sharp lower estimate for the expectation of the mean width
of a (symmetric) random polytope inside an isotropic convex body, for the
range $N\simeq n$. We will use their result in the proof of Theorem
\ref{THM estimate high brob} in order to show Theorem \ref{THM estimate
expectation}.

\medskip

The paper is organized as follows. In Section \ref{SEC preli} we state
additional notation and known results that we will use throughout the
paper. Section \ref{SEC lemmas} is devoted to show several technical
lemmas which will be needed in the proofs of the theorems. Finally, in
Section \ref{SEC main} we prove Theorems \ref{THM estimate high brob} and
\ref{THM estimate expectation}. We conclude the paper showing in Section
\ref{SEC gaussian} the following result, in the flavor of Theorem \ref{THM
estimate expectation}, for Gaussian random polytopes:
\begin{thm}\label{thm gaussian}
Let $X_1,\dots,X_N$ be independent standard Gaussian random vectors in
$\R^n$, $n\leq N$, and let $K_N=\conv\{X_1,\dots,X_N\}$. Then, for all
$1\leq k\leq n$,
\[
\E\widetilde{R}_k(K_N)\simeq\max\left\{\sqrt k,\sqrt{\log N}\right\}.
\]
\end{thm}

\section{Additional notation and background results}\label{SEC preli}

The following definition can be found in \cite{P1} (see also \cite{P2}). Let $K$ be
a convex body in ${\mathbb R}^n$ with volume one. For every
$-n<p<+\infty$, $p\neq 0$, and every $F\in G_{n,k}$, the $p$-moment of the
Euclidean norm of the projection of $K$ onto $F$ is defined as
\begin{equation*}
I_p(K,F)=\left(\int_K|P_Fx|^p\,dx\right)^{1/p}.
\end{equation*}
If $F=\R^n$ we write
\begin{equation*}
I_p(K)=I_p(K,\R^n)=\left(\int_K|x|^p\,dx\right)^{1/p}.
\end{equation*}
In \cite{P1}, Paouris proved a sharp reverse H\"{o}lder type inequality
for the moments of the Euclidean norm of an isotropic convex body $K$ in
${\mathbb R}^n$:
\begin{equation}\label{eq.Paouris1}
I_q(K)\simeq I_2(K)=\sqrt{n}L_K,\quad\text{ for all }\;1\leq
q\leq\sqrt{n}.
\end{equation}
This inequality provides a sharp concentration result for the mass
distribution in an isotropic convex body in $\R^n$, {\it i.e.}, for any $t>0$,
\[
\Pro\bigl(|x|\geq ct\sqrt{n}\,L_K\bigr)\leq e^{-\sqrt{n}\,t}.
\]
We will need this result for the proof of the upper bound in Theorem
\ref{THM estimate high brob}.

In \cite{P2}, the same author got the following small ball probability
result for isotropic random vectors, {\it i.e.}, centered random vectors
$X$ verifying that $\E\langle X,\theta\rangle^2=1$ for any $\theta\in
S^{n-1}$:
\begin{thm}[Paouris, \cite{P2}]
Let $X$ be an isotropic log-concave random vector in $\R^n$. Let $M$ be a
non-zero $n\times n$ matrix, $y\in\R^n$ and $\varepsilon\in (0,1)$. Then
\[
\Pro\bigl(|MX-y|\leq c_1\varepsilon\Vert M\Vert_{{\rm HS}}\bigr)\leq
\varepsilon^{c_2\frac{\Vert M\Vert_{{\rm HS}}}{\Vert M\Vert_{{\rm op}}}},
\]
where $c_1,c_2$ are positive absolute constants and $\Vert M\Vert_{{\rm
HS}},\Vert M\Vert_{{\rm op}}$ denote the Hilbert-Smith and the operator
norm of $M$, respectively.
\end{thm}

In order to show the theorem, the author proved a reverse H\"{o}lder
inequality for the negative moments of the Euclidean norm on an isotropic
convex body $K\subset\R^n$:
\begin{equation}\label{eq.Paouris2}
I_{-q}(K)\simeq I_2(K)=\sqrt{n}\,L_K, \quad\text{ for all }\;1\leq
q\leq\sqrt{n}.
\end{equation}

For the proof of the lower bound in Theorem \ref{THM estimate high brob}
we need some more background. For a convex body $K\subset\R^n$,
$0\in\inter K$, and every $-\infty<p<+\infty$, $p\neq 0$, we define
\begin{equation*}\label{def.p-meanw idth}
w_p(K)=\left(\int_{S^{n-1}}h_K(\theta)^p\,d\sigma(\theta)\right)^{1/p},
\end{equation*}
where $h_K(\theta)=\max_{x\in K}\inner{x}{\theta}$ denotes the support
function of $K$ in the direction $\theta\in S^{n-1}$. If $|K|=1$, then for
all $q\geq 1$, the $L_q$-centroid body of $K$ is defined to be the
symmetric convex body in $\R^n$ whose support function is given by
\begin{equation*}\label{def.Zq}
h_{Z_q(K)}(\theta)=\left(\int_K\bigl|\inner{x}{\theta}\bigr|^q\,dx\right)^{1/q}.
\end{equation*}
In \cite{P2}, Paouris also established the following asymptotic formula
for the negative moments $I_{-q}(K,F)$ (see \cite[Propositions~5.4 and
4.1]{P2}):
\begin{proposition}[Paouris, \cite{P2}]\label{prop.paouris3}
Let $K\subset\R^n$ be a convex body with $|K|=1$ and centroid at the
origin. Let $k\in\{1,\dots,n\}$ and $F\in G_{n,k}$. Then, for every
integer $q<k$ it holds
\begin{equation*}\label{eq.Paouris3}
I_{-q}(K,F) \simeq \sqrt{\frac{k}{q}}\,w_{-q}\bigl(P_FZ_q(K)\bigr).
\end{equation*}
\end{proposition}

\section{Some preliminary lemmas}\label{SEC lemmas}

We state here some preliminary technical results which will be needed in
the proofs of the main theorems. The first observation states the
monotonicity of the family of mean outer radii $\widetilde{R}_k$ in $k$.

\begin{lemma}\label{lemma.increasing}
Let $K\subset\R^n$ be a convex body. Then $\widetilde{R}_{k-1}(K)\leq
\widetilde{R}_k(K)$, $2\leq k\leq n$.
\end{lemma}

\begin{proof}
Let $F\in G_{n,k}$, with $k\geq 2$. Then, for any ($k-1$)-dimensional
subspace $E$ of $F$, we have that $R(P_EK)\leq R(P_FK)$. Thus
\[
\int_{G_{F,k-1}}R(P_EK)\,d\nu_{F,k-1}(E)\leq R(P_FK),
\]
and hence
\[
\int_{G_{n,k}}\int_{G_{F,k-1}}R(P_EK)\,d\nu_{F,k-1}(E)\,d\nu_{n,k}(F)\leq
\int_{G_{n,k}}R(P_FK)\,d\nu_{n,k}(F)=\widetilde{R}_k(K).
\]
By the uniqueness of the Haar probability measure on $G_{n,k-1}$, the
integral on the left hand side of the previous inequality is
\[
\int_{G_{n,k-1}}R(P_EK)\,d\nu_{n,k-1}(E)=\widetilde{R}_{k-1}(K).\qedhere
\]
\end{proof}

The following lemmas will be needed to estimate from above the mean outer
radii of $K_N$ with high probability, for which we will follow the
arguments from \cite{DGT1} and \cite{DGT2}, used there to estimate the
normalized quermassintegrals of $K_N$.

\begin{lemma}\label{lemma3}
Let $K\subset\R^n$ be an isotropic convex body and $K_N$ be a random
polytope generated by $N\geq n$ independent random vectors $X_1,\dots,X_N$ uniformly distributed in $K$.
Then, for all $1\leq k \leq n$, $t>1$
and $q \geq \log N$,
\[
\widetilde{R}_k(K_N)\leq
ct\left(\int_{G_{n,k}}I_q(K,F)^q\,d\nu_{n,k}(F)\right)^{1/q}
\]
with probability greater than $1-t^{-q}$, where $c>0$ is an absolute
constant.
\end{lemma}

\begin{proof}
Applying H\"{o}lder's  inequality and Cauchy-Schwartz's inequality we get
\[
\begin{split}
\widetilde{R}_k(K_N) & =\int_{G_{n,k}}R(P_FK_N)\,d\nu_{n,k}(F)
    \leq\left(\int_{G_{n,k}}R(P_FK_N)^{q/2}\,d\nu_{n,k}(F)\right)^{2/q}\\
 & \leq\left(\int_{G_{n,k}}I_q(K,F)^q\,d\nu_{n,k}(F)\right)^{1/q}
  \left(\int_{G_{n,k}}\frac{R(P_FK_N)^q}{I_q(K,F)^q}\,d\nu_{n,k}(F)\right)^{1/q}.
\end{split}
\]
So, in order to conclude the proof, it suffices to show that
\begin{equation}\label{eq2}
\left(\int_{G_{n,k}}\frac{R(P_FK_N)^q}{I_q(K,F)^q}
\,d\nu_{n,k}(F)\right)^{1/q}\leq ct
\end{equation}
with probability greater than $1-t^{-q}$, where $c>0$ is an absolute
constant.

In order to prove \eqref{eq2} we notice that, on the one hand, since $K$
is isotropic then for every $F\in G_{n,k}$ it holds $R(K)\leq (n+1)L_K$
(see, for instance, \cite{KLS}), and together with the definition of isotropic constant and
H\"older's inequality we get
\[
R(P_FK_N)\leq R(K)\leq (n+1)L_K\leq\frac{n+1}{\sqrt{k}}I_q(K,F).
\]
Therefore,
\begin{equation}\label{eq1}
\begin{split}
\E\int_{G_{n,k}} & \frac{R(P_FK_N)^q}{I_q(K,F)^q}\,d\nu_{n,k}(F)\\
  & \qquad=\E\int_0^{+\infty}qt^{q-1}\nu_{n,k}\bigl\{F\in G_{n,k}:R(P_FK_N)\geq t\,I_q(K,F)\bigr\}\,dt\\
  & \qquad=\int_0^{\frac{n+1}{\sqrt k}}qt^{q-1}\E\nu_{n,k}\bigl\{F\in G_{n,k}:R(P_FK_N)\geq t\,I_q(K,F)\bigr\}\,dt.
\end{split}
\end{equation}
On the other hand, since $R(P_FK_N)=\max_{1\leq j \leq N}|P_FX_j|$ for
every $F\in G_{n,k}$, using Markov's inequality we get that for any $1\leq
j\leq N$,
\[
\Pro\bigl(|P_FX_j|\geq tI_q(K,F)\bigr)\leq t^{-q},
\]
and thus the union bound gives
\begin{equation*}
\Pro\bigl(R(P_FK_N)\geq tI_q(K,F)\bigr)\leq N t^{-q}.
\end{equation*}
By a standard application of Fubini's theorem we get that
\[
\begin{split}
\E\nu_{n,k}\bigl\{F\in G_{n,k}: & \,R(P_FK_N)\geq tI_q(K,F)\bigr\}\\
 & =\int_{G_{n,k}}\Pro\bigl(R(P_FK_N)\geq tI_q(K,F)\bigr)\,d\nu_{n,k}(F)\leq Nt^{-q}.
\end{split}
\]
This, together with \eqref{eq1} show that for any
$A\in\bigl(0,(n+1)/\sqrt{k}\bigr)$ it holds
\[
\begin{split}
\E\int_{G_{n,k}}\frac{R(P_FK_N)^q}{I_q(K,F)^q}\,d\nu_{n,k}(F) &
  \leq\int_0^Aqt^{q-1}\,dt+\int_A^{\frac{n+1}{\sqrt k}}qt^{q-1}Nt^{-q}\,dt\\
 & =A^q + q N \log\left(\frac{n+1}{A\sqrt k}\right).
\end{split}
\]
Taking $A=e$ and $q\geq \log N$ we obtain that this is bounded from above
by $c^q$ with $c$ an absolute constant. Inequality \eqref{eq2} now follows
from Markov's inequality. It completes the proof of the lemma.
\end{proof}

\begin{lemma}\label{lemma4}
Let $K\subset\R^n$ be a convex body, $1\leq k\leq n$ and $q\geq 1$. Then
\[
\left(\int_{G_{n,k}}I_q(K,F)^q\,d\nu_{n,k}(F)\right)^{1/q}\simeq\sqrt{\frac{k+q}{n+q}}\,I_q(K).
\]
\end{lemma}

\begin{proof}
In \cite{P0} it was shown that for any $x\in\R^n$ and all $q\geq 1$ it
holds
\begin{equation*}\label{eq.2-norm}
|x|\simeq\sqrt{\frac{n+q}{q}}\left(\int_{S^{n-1}}\bigl|\inner{x}{\theta}\bigr|^q\,d\sigma(\theta)\right)^{1/q}.
\end{equation*}
Applying Fubini's theorem, the above formula and the uniqueness of the Haar measure on $S^{n-1}$ we obtain
\[
\begin{split}
\Biggl(\int_{G_{n,k}}I_q(K, & F)^q\,d\nu_{n,k}(F)\Biggr)^{1/q}
    =\left(\int_K \int_{G_{n,k}} |P_Fx|^q\,d\nu_{n,k}(F)\,dx\right)^{1/q}\\
 & \simeq\sqrt{\frac{k+q}{q}}\left(\int_K \int_{G_{n,k}}\int_{S_F^{k-1}}\bigl|\inner{P_Fx}{\theta}\bigr|^q
    \,d\sigma_F(\theta)\,d\nu_{n,k}(F)\,dx\right)^{1/q}\\
 & =\sqrt{\frac{k+q}{q}}\left(\int_K \int_{G_{n,k}}\int_{S_F^{k-1}}\bigl|\inner{x}{\theta}\bigr|^q
    \,d\sigma_F(\theta)\,d\nu_{n,k}(F)\,dx\right)^{1/q}\\
 & =\sqrt{\frac{k+q}{q}}\left(\int_K\int_{S^{n-1}}\bigl|\inner{x}{\theta}\bigr|^q
    \,d\sigma_{n-1}(\theta)\,dx\right)^{1/q}\\
 & \simeq\sqrt{\frac{k+q}{n+q}}\left(\int_K|x|^q\,dx \right)^{1/q}
  =\sqrt{\frac{k+q}{n+q}}\;I_q(K).\qedhere
\end{split}
\]
\end{proof}

The following lemmas will be used for estimating the lower bound in
Theorem~\ref{THM estimate high brob}. The method will be similar to the
one for the upper bound, but involving the negative moments of the
Euclidean norm.

\begin{lemma}\label{lemma.low2}
Let $K\subset\R^n$ be a convex body and $K_N$ be a random polytope
generated by $N\geq n$ independent random vectors $X_1,\dots,X_N$
uniformly distributed in $K$. Then, for any $1\leq k\leq n$, $q<k$ and
$t\in(0,1)$,
\begin{equation*}
\left(\int_{G_{n,k}}\frac{R(P_FK_N)^{-q}}{I_{-q}(K,F)^{-q}}\,d\nu_{n,k}(F)\right)^{-1/q}
\geq t\,\left(\frac{N-1}{N}\right)^{1/q}
\end{equation*}
with probability greater than $1-t^q$.
\end{lemma}

\begin{proof}
Since $R(P_FK_N)=\max_{1\leq j \leq N}|P_FX_j|$ for every $F\in G_{n,k}$,
using Markov's inequality we get that for any $1\leq j\leq N$ and all
$q<k$,
\[
\Pro\bigl(|P_FX_j|\leq\varepsilon I_{-q}(K,F)\bigr)=
\Pro\bigl(|P_FX_j|^{-q}\geq\varepsilon^{-q}\,I_{-q}(K,F)^{-q}\bigr)\leq\varepsilon^q,
\]
for every $\varepsilon\in(0,1)$, and thus
\begin{equation*}
\Pro\bigl(R(P_FK_N)\leq\varepsilon I_{-q}(K,F)\bigr)\leq\varepsilon^{Nq}.
\end{equation*}
Then, a standard application of Fubini's theorem leads to
\begin{equation*}\label{eq.strong.sbp}
\E\nu_{n,k}\bigl\{F\in G_{n,k}:R(P_FK_N)\leq\varepsilon I_{-q}(K,F)\bigr\}
\leq\varepsilon^{Nq},\quad \varepsilon\in(0,1),
\end{equation*}
which can be used to bound the expectation from above:
\[
\begin{split}
\E\int_{G_{n,k}} & \frac{R(P_F K_N)^{-q}}{I_{-q}(K,F)^{-q}}\,d\nu_{n,k}(F)\\
 & \qquad=\E\int_0^{+\infty}qt^{q-1}\nu_{n,k}\bigl\{F\in G_{n,k}:I_{-q}(K,F)\geq tR(P_FK_N)\bigr\}\,dt\\
 & \qquad=\int_0^{+\infty}qt^{q-1}\E\nu_{n,k}\left\{F\in G_{n,k}:R(P_FK_N)\leq\frac{1}{t}I_{-q}(K,F)\right\}\,dt\\
 & \qquad\leq\int_0^1 qt^{q-1}\,dt+\int_1^{+\infty}qt^{q-1-Nq}\,dt=1+\frac1{N-1}=\frac{N}{N-1}.
\end{split}
\]
To finish the proof we apply Markov's inequality.
\end{proof}

In the next lemma, Proposition \ref{prop.paouris3} plays a crucial role.

\begin{lemma}\label{lemma.low3}
Let $K\subset\R^n$ be a convex body with $|K|=1$ and centroid at the
origin. For every $1\leq k \leq n$ and every $q<k$ it holds
\[
\left(\int_{G_{n,k}}I_{-q}(K,F)^{-q}\,d\nu_{n,k}(F)\right)^{-1/q}
\simeq\sqrt{\frac{k}{n}}\,I_{-q}(K).
\]
\end{lemma}

\begin{proof}
Using Proposition \ref{prop.paouris3} and the uniqueness of the Haar
probability measure on the sphere $S^{n-1}$, we get
\[
\begin{split}
\Biggl(\int_{G_{n,k}}\! I_{-q}(K,F)^{-q} & d\nu_{n,k}(F)\Biggr)^{-1/q}
    \!\!\simeq\sqrt{\frac{k}{q}}\left(\int_{G_{n,k}}\!w_{-q}\bigl(P_FZ_q(K)\bigr)^{-q}\,d\nu_{n,k}(F)\right)^{-1/q}\\
 & =\sqrt{\frac{k}{q}}\left(\int_{G_{n,k}}\int_{S_F^{k-1}}h_{P_FZ_q(K)}(\theta)^{-q}
    \,d\sigma_F(\theta)\,d\nu_{n,k}(F)\right)^{-1/q}\\
 & =\sqrt{\frac{k}{q}}\left(\int_{G_{n,k}}\int_{S_F^{k-1}}h_{Z_q(K)}(\theta)^{-q}
    \,d\sigma_F(\theta)\,d\nu_{n,k}(F)\right)^{-1/q}\\
 & =\sqrt{\frac{k}{q}}\left(\int_{S^{n-1}}h_{Z_q(K)}(\theta)^{-q}\,d\sigma(\theta)\right)^{-1/q}
    =\sqrt{\frac{k}{q}}\,w_{-q}\bigl(Z_q(K)\bigr).
\end{split}
\]
Applying again Proposition \ref{prop.paouris3} we get the result.
\end{proof}

\section{Asymptotic Formula for the Mean Outer Radii of Random Polytopes} \label{SEC main}

Now we are ready to prove Theorem \ref{THM estimate high brob}.

\begin{proof}[Proof of Theorem \ref{THM estimate high brob}]
First we state the upper bound for the mean outer radii of $K_N$ in the
full range $n\leq N\leq e^{\sqrt n}$. Applying Lemma \ref{lemma3} and
Lemma \ref{lemma4} with $q=\log N$ and $t=e^s$ for any $s>0$, we get that
\[
\widetilde{R}_k(K_N)\leq c(s)\sqrt{\frac{k+\log N}{n+\log N}}\;I_{\log
N}(K)
\]
with probability greater than $1-N^{-s}$, where $c(s)>0$ is an absolute
constant depending only on $s$. Using (\ref{eq.Paouris1}) for $q=\log N$,
and since $N\leq e^{\sqrt{n}}$, we obtain that $I_{\log N}(K) \leq c
\sqrt{n}\,L_K$. Consequently, we get
\[
\widetilde{R}_k(K_N)\leq c(s)\sqrt{k+\log N}\,L_K\simeq
\max\left\{\sqrt{k},\sqrt{\log N}\right\}\,L_K.
\]

Next we show the lower bound in Theorem \ref{THM estimate high brob}. We start assuming that $\log
N<k$, $1\leq k \leq n$. By H\"older's inequality,
\[
\begin{split}
\widetilde{R}_k(K_N) & =\int_{G_{n,k}}R(P_FK_N)\,d\nu_{n,k}(F)
    \geq\left(\int_{G_{n,k}}R(P_FK_N)^{-q/2}\,d\nu_{n,k}(F)\right)^{-2/q}\\
 & \geq\left(\int_{G_{n,k}}I_{-q}(K,F)^{-q}\,d\nu_{n,k}(F)\right)^{-1/q}\!\!
  \left(\int_{G_{n,k}}\frac{R(P_FK_N)^{-q}}{I_{-q}(K,F)^{-q}}\,d\nu_{n,k}(F)\right)^{-1/q}\!\!,
\end{split}
\]
and then Lemmas \ref{lemma.low2} and \ref{lemma.low3} for $q=\log N<k$ and
$t=e^{-s}$ imply that
\begin{equation}\label{lowerboundsqrtk}
\widetilde{R}_k(K_N)\geq c(s)\sqrt{\frac{k}{n}}\,I_{-\log N}(K)
\end{equation}
with probability greater than $1-N^{-s}$, where $c(s)$ is an absolute
constant depending only on $s$. Since $\log N\leq\sqrt{n}$, using
(\ref{eq.Paouris2}) we get that $I_{-\log N}(K)\simeq\sqrt{n}\,L_K$ and
the result follows.

In \cite{DGT2} it was shown that $\bigl(|K_N|/|B_2^n|\bigr)^{1/n}\geq
c(s)\sqrt{\log(N/n)}\,L_K$ with probability greater than $1-N^{-s}$. The
monotonicity of the mean outer radii (Lemma \ref{lemma.increasing}) and
the fact that $\bigl(|K_N|/|B_2^n|\bigr)^{1/n}\leq \widetilde{R}_1(K_N)$
complete the proof of this case.
\end{proof}

\begin{rmk}\label{rem_extending}
In order to obtain the upper bound for the mean outer radii in
Theo\-rem~\ref{THM estimate high brob} we cannot expect to extend the
range of $N$ to $N=e^{\beta n}$ for some $\beta\in(0,1)$. In fact, if this
holds for, say, $\beta=1/2$, then in the case $k=1$ we get that for every
isotropic convex body $K\subset\R^n$, we have
\[
\Pro\bigl(\omega(K_N)\leq \widetilde{R}_1(K_N)\leq\sqrt{n}L_K\bigr)\geq
1-e^{-n/2}.
\]
Since it is known (see \cite{GM}) that for $N\simeq e^{n/2}$ it holds
\[
\Pro\bigl(K_N\subseteq K\subseteq c(\delta)K_N\bigr)\geq 1-\delta,
\]
then choosing the parameters ($\delta > 1/2$, $n\geq 2$) we get that the
event
\[
\bigl\{\omega(K_N)\leq \sqrt n \,L_K \; \text{ and }\;K\subseteq
cK_N\bigr\}
\]
occurs with positive probability. Thus, for every isotropic convex body
$K\subset\R^n$, $\omega(K)\leq c\,\omega(K_N)\leq C\,\sqrt n\,L_K$, which
is not true, as the cross-polytope shows.

Regarding an extension of the range of $N$ from below, we note that
similar arguments work also for $n^{1+\delta}<N<e^{\sqrt{n}}$, for any
$\delta>0$, just replacing the constants $c_1(\delta)$ and $c_2(\delta)$.
However, if the number of vertices $N$ is proportional to the dimension
$n$, {\it i.e.}, $N \simeq n$, more refined arguments are needed.
\end{rmk}

Now we show that if the expectation is involved, the range for $N$ can be
extended from below till $N\geq n$. In the proof we
will need the following result which was shown in \cite{AP2}:
in our notation, if $K\subset\R^n$ is an isotropic convex body and
$X_1,\dots,X_N$ are independent random vectors uniformly distributed in
$K$, then for every $n\leq N\leq e^{\sqrt{n}}$ we have
\begin{equation}\label{David-Joscha}
\E\widetilde{R}_1(K_N)\simeq \sqrt{\log N}\,L_K.
\end{equation}

\begin{rmk}
Since in \cite{AP2} the authors were interested in studying the value of
the mean width of $K_N$, which is the same as $\widetilde{R}_1(K_N)$ only
when $K_N$ is symmetric, equation (\ref{David-Joscha}) is specifically
written only in the symmetric case. However, the same proof leads to this
equation in the non-symmetric case.
\end{rmk}

\begin{proof}[Proof of Theorem \ref{THM estimate expectation}]
First we use the fact that the upper bound in Theorem \ref{THM estimate
high brob} is valid for the whole range $n\leq N\leq e^{\sqrt n}$. Thus,
for every $1\leq k\leq n$ and all $s>0$,
\[
\E\widetilde{R}_k(K_N)\leq c(s)\max\left\{\sqrt{k},\sqrt{\log
N}\right\}L_K +N^{-s}R(K).
\]
Since $R(K)\leq (n+1)L_K$ and
taking $s=1$, we obtain, on the one hand, that
\begin{equation}\label{eq.up expect}
\E\widetilde{R}_k(K_N)\leq c\max\left\{\sqrt{k},\sqrt{\log N}\right\}L_K.
\end{equation}

On the other hand, by Markov's inequality, we have that for every $\alpha>0$
$$
\E\widetilde{R}_k(K_N)\geq\alpha\Pro(\widetilde{R}_k(K_N)\geq\alpha).
$$
From equation (\ref{lowerboundsqrtk}) we have
\begin{equation}\label{eq.low expect k}
\E\widetilde{R}_k(K_N)\geq c(s)\sqrt{k}L_K\left(1-N^{-s}\right)\geq c\sqrt L_K
\end{equation}
choosing $s$ to be some absolute constant.
Finally, using (\ref{David-Joscha}) and the
monotonicity of the mean outer radii (Lemma \ref{lemma.increasing}), we
have that, for all $1\leq k\leq n$ and $n<N<e^{\sqrt{n}}$,
\begin{equation}\label{eq.low expect logN}
\E\widetilde{R}_k(K_N)\geq
\E\widetilde{R}_1(K_N)\simeq\sqrt{\log N}\,L_K.
\end{equation}
The theorem follows from (\ref{eq.up expect}), (\ref{eq.low expect k}) and
(\ref{eq.low expect logN}).
\end{proof}

\section{The Gaussian case}\label{SEC gaussian}

In this last section we consider the case of Gaussian random polytopes,
and show Theorem \ref{thm gaussian}. We observe that, since the
distribution of a Gaussian vector is rotationally invariant and its
projection on any $k$-dimensional subspace is a $k$-dimensional Gaussian
vector, it is a direct consequence of the following result for the
orthogonal projection of Gaussian random vectors:

\begin{proposition}\label{prop gaussian}
Let $G_1,\dots,G_N$ be independent standard Gaussian random vectors in
$\R^k$. Then
\[
\E\max_{1\leq j\leq N}|G_j|\simeq\max\left\{\sqrt k,\sqrt{\log
N}\right\}.
\]
\end{proposition}

Although this result is probably known by specialists, we include a proof,
since we were not able to find one in the literature except for the
1-dimensional case (see \cite{LT}). First we need the following lemma:

\begin{lemma}
Let $k\geq 1$ and $t\geq\max\bigl\{\sqrt{2(k-1)},1\bigr\}$. Then
$$
t^{k-1}e^{-\frac{t^2}{2}}\leq\int_t^\infty
r^ke^{-\frac{r^2}{2}}\,dr\leq 2t^{k-1}e^{-\frac{t^2}{2}}.
$$
\end{lemma}
\begin{proof}
Since $t\geq 1$
\[
\int_t^\infty r^ke^{-\frac{r^2}{2}}\,dr\geq t^{k-1}\int_t^\infty
re^{-\frac{r^2}{2}}\,dr=t^{k-1}e^{-\frac{t^2}{2}},
\]
which shows the left inequality. In order to prove the right hand side inequality,
we consider the function
\[
f(t)=2t^{k-1}e^{-\frac{t^2}{2}}-\int_t^\infty r^ke^{-\frac{r^2}{2}}dr.
\]
Since $f^\prime(t)=t^{k-2}e^{-\frac{t^2}{2}}\bigl(2(k-1)-t^2\bigr)$,
$f(t)$ decreases if $t\geq \sqrt{2(k-1)}$ which, together with the fact
that $\lim_{t\to\infty}f(t)=0$, gives the result.
\end{proof}

\begin{proof}[Proof of Proposition \ref{prop gaussian}]
For $k=1$ the result is well known, so we assume that $k\geq 2$. Integrating in polar coordinates
and taking into account that $k\geq 2$, we have
\[
\begin{split}
\Pro\left(\max_{1\leq j\leq N}|G_j|\leq t\right) &
    =\Pro\bigl(|G_j|\leq t\bigr)^N
    =\left(k|B_2^k|\int_0^tr^{k-1}\frac{e^{-\frac{r^2}{2}}}{\bigl(\sqrt{2\pi}\bigr)^k}\,dr\right)^N\\
 & =\left(1-k|B_2^k|\int_t^\infty r^{k-1}\frac{e^{-\frac{r^2}{2}}}{\bigl(\sqrt{2\pi}\bigr)^k}\,dr\right)^N\\
 & \leq\left(1-k|B_2^k|t^{k-2}\int_t^\infty r\frac{e^{-\frac{r^2}{2}}}{\bigl(\sqrt{2\pi}\bigr)^k}\,dr\right)^N\\
 & =\left(1-k|B_2^k|t^{k-2}\frac{e^{-\frac{t^2}{2}}}{\bigl(\sqrt{2\pi}\bigr)^k}\right)^N
    =e^{N\log\left(1-k|B_2^k|t^{k-2}\frac{e^{-\frac{t^2}{2}}}{(\sqrt{2\pi})^k}\right)}\\
 & \leq e^{-Nk|B_2^k|t^{k-2} \frac{e^{-\frac{t^2}{2}}}{(\sqrt{2\pi})^k}}.
\end{split}
\]
Now, using the well-known value of the volume of the Euclidean unit ball,
namely, $|B_2^k|=\pi^{k/2}/\Gamma\bigl(1+k/2\bigr)$, by Stirling's formula
we have that there exists an absolute constant $c$ such that
\[
\Pro\left(\max_{1\leq j\leq N}|G_j|\leq t\right)\leq
e^{-c\frac{N\sqrt{k}}{t^2}\left(\frac{\sqrt et}{\sqrt k}\right)^k
e^{-\frac{t^2}{2}}}.
\]
Thus, taking $t=\sqrt k$ we obtain that
\[
\Pro\left(\max_{1\leq j\leq N}|G_j|\leq \sqrt k\right)\leq
e^{-c\frac{N}{\sqrt k}}\leq e^{-c\sqrt n},
\]
which tends to 0 when $n\to\infty$. Hence, there exists $n_0\in \N$ such
that if $n\geq n_0$,
\[
\Pro\left(\max_{1\leq j\leq N}|G_j|>\sqrt k\right)\geq\frac{1}{2}
\]
and so, if $n\geq n_0$,
\[
\E\max_{1\leq j\leq N}|G_j|\geq\sqrt{k}\,\Pro\left(\max_{1\leq j\leq
N}|G_j|\geq\sqrt k\right)>\frac{\sqrt k}{2}.
\]
This shows the lower estimate for the expectation when $k\geq e\log N$. Otherwise, taking
$t=\sqrt{\log N}$ we obtain
\[
\Pro\left(\max_{1\leq j\leq N}|G_j|\leq \sqrt{\log N}\right)\leq
e^{-c\frac{\sqrt N\sqrt{k}}{\log N}\left(\frac{\sqrt{e\log N}}{\sqrt
k}\right)^k }\leq e^{-c\frac{\sqrt N}{\log N}},
\]
because $1\leq k\leq e\log N$. Thus, there exists $n_0\in \N$ such that if
$n\geq n_0$,
\[
\Pro\left(\max_{1\leq j\leq N}|G_j|\geq \sqrt{\log N}\right)
\geq\frac{1}{2}
\]
and therefore, if $n\geq n_0$ and $1\leq k\leq e\log N$,
\[
\E\max_{1\leq j\leq N}|G_j|\geq\sqrt{\log N}\,\Pro\left(\max_{1\leq
j\leq N}|G_j|\geq\sqrt{\log N}\right)\geq\frac{\sqrt{\log N}}{2}.
\]
On the other hand, for any $A>0$,
\[
\begin{split}
\E\max_{1\leq j\leq N}|G_j| &
    =\int_0^\infty\Pro\left(\max_{1\leq j\leq N}|G_j|\geq t\right)\,dt\\
 & =\int_0^A\Pro\left(\max_{1\leq j\leq N}|G_j|\geq t\right)\,dt
    +\int_A^\infty\Pro\left(\max_{1\leq j\leq N}|G_j|\geq t\right)\,dt\\
 & \leq A+N\int_A^\infty\Pro\bigl(|G_1|\geq t\bigr)\,dt\\
 & =A+Nk\,|B_2^k|\int_A^\infty\int_t^\infty r^{k-1}\frac{e^{-\frac{r^2}{2}}}{\bigl(\sqrt{2\pi}\bigr)^k}\,dr\,dt.
\end{split}
\]
Taking $A=\max\left\{2\sqrt{\log N},\sqrt{2k}\right\}$ we get that
\[
\begin{split}
\E\max_{1\leq j\leq N}|G_j| &
    \leq A+2Nk\,|B_2^k|\int_A^\infty t^{k-2}\frac{e^{-\frac{t^2}{2}}}{\bigl(\sqrt{2\pi}\bigr)^k}\,dt\\
 & \leq A+4Nk\,|B_2^k|A^{k-3}\frac{e^{-\frac{A^2}{2}}}{\bigl(\sqrt{2\pi}\bigr)^k}
    \leq A+ \frac{CN\sqrt k}{A^3}\left(\frac{\sqrt e A}{\sqrt{k}}\right)^ke^{-\frac{A^2}{2}}.
\end{split}
\]
If $A=2\sqrt{\log N}$ it can be checked that
\[
\E\max_{1\leq j\leq N}|G_j|\leq 2\sqrt{\log N}+\frac{C}{N\log N}
2\sqrt{e}\sqrt{\log N}\leq C\sqrt{\log N}.
\]
Finally, if $A=\sqrt{2k}$, then
\[
\E\max_{1\leq j\leq N}|G_j|
\leq\sqrt{2k}+\frac{CN}{k}e^{-\frac{k}{2}}\leq\sqrt{2k}+\frac{C}{k}\leq
C\sqrt{k}.\qedhere
\]
\end{proof}

\end{document}